\documentclass[12pt, reqno]{amsart}

\usepackage{amscd,amsthm,amsfonts,amssymb,amsmath,esint}
\usepackage{amsmath,tikz-cd}
\usepackage[colorlinks=true]{hyperref}
\usepackage{fullpage}
\usepackage[all]{xy}
\usepackage{mathtools}
\usepackage{mathrsfs}
\usepackage[pagewise]{lineno}

    \newcommand{\BA}{{\mathbb {A}}} 
    \newcommand{\BC}{{\mathbb {C}}}

     \newcommand{\BR}{{\mathbb {R}}}

     \newcommand{\BZ}{{\mathbb {Z}}}

     \newcommand{\CL}{{\mathcal {L}}}

    \newcommand{\CS}{{\mathcal {S}}}

    \newcommand{\RI}{{\mathrm {I}}} 

    \newcommand{\RO}{{\mathrm {O}}}

    \newcommand{\RW}{{\mathrm {W}}}

    \newcommand{\ad}{{\mathrm{ad}}}

    \newcommand{\diag}{{\mathrm{diag}}}

    \newcommand{\Gal}{{\mathrm{Gal}}} \newcommand{\GL}{{\mathrm{GL}}}
    
    \newcommand{\Hom}{{\mathrm{Hom}}}

    \newcommand{\GO}{{\mathrm{GO}}}

    \newcommand{\SL}{{\mathrm{SL}}}
     
    \newcommand{\Sp}{{\mathrm{Sp}}}

      \newcommand{\Supp}{{\mathrm{Supp}}}

    \newcommand{\Vol}{{\mathrm{Vol}}}

    \newcommand{\GSO}{{\mathrm{GSO}}}

    \newcommand{\sk}{\medskip}
    \newcommand{\lra}{\longrightarrow}
     
    \newcommand{\bs}{\backslash}
    
    \newcommand{\s}{\sk\noindent}

    \theoremstyle{plain}
    \newtheorem{thm}{Theorem}[section] 
    \newtheorem{lem}[thm]{Lemma}  \newtheorem{prop}[thm]{Proposition}

 \newtheorem{def-prop}[thm]{Definition-Proposition}

\theoremstyle{remark} \newtheorem{remark}[thm]{Remark}
\theoremstyle{remark} 
\theoremstyle{remark} 

    \numberwithin{equation}{section}

\begin{document}

\title{Ichino Periods for CM Forms}

\author{Li Cai}
\address{Academy for Multidisciplinary Studies\\
Beijing National Center for Applied Mathematics\\
Capital Normal University\\
Beijing, 100048, People's Republic of China}
\email{caili@cnu.edu.cn}

\author{Yangyu Fan}
\address{Academy for Multidisciplinary Studies\\
Beijing National Center for Applied Mathematics\\
Capital Normal University\\
Beijing, 100048, People's Republic of China}
\email{b452@cnu.edu.cn}

\author{Yong Jiang}
\address{School of Mathematical Sciences\\
University of Science and Technology of China\\
Hefei, Anhui 230026, People's Republic of China}
\email{jiangyong@ustc.edu.cn}

\begin{abstract}
In both local and global settings, we establish explicit relations between Ichino triple product period and Waldspurger toric periods 
for CM forms via the theta lifting and the see-saw principle.
\end{abstract}

\maketitle

\tableofcontents

\section{Introduction}
Let $F$ be a number field with $\BA$ its ring of adeles. 
Let $H \subset G$ be reductive groups over $F$. A basic object in the theory
of automorphic forms is the {\em (H-)period integrals}
\[ \int_{[H]} \varphi(h)dh\]
for automorphic forms  $\varphi$ on $G(\BA)$. Here, $[H] = H(F) \bs H(\BA)$ is the automorphic quotient of $H$.

Period integrals are closely related to $L$-values. When $(G,H)$ is {\em spherical} (that is,
there is a Borel subgroup $B\subset G$ such that $BH$ is open in $G$), a general conjecture of 
Sakellaridis-Venkatesh \cite{SV} states roughly that for an irreducible unitary automorphic representation $\pi=\otimes_v^\prime\pi_v$ on $G(\BA)$, the square of the $H$-period integral on $\pi$ decomposes into  the product of certain $L$-values associated to $\pi$ and the so-called canonical local period integrals on $\pi_v$, which are closely related to the local $L$-value of $\pi_v$.

When two (products of) global/local period integrals are related to a common $L$-value, it seems natural to expect  there exists an {\em explicit} relation between these two global/local period integrals. In this paper, we shall establish such a {\em period relation} between the Waldspurger toric period integral and the Ichino triple product period integral via theta lifting and the see-saw principle.

More precisely, let $\pi$ be an unitary cuspidal automorphic representation on $\GL_2(\BA)$ with central character $\omega$. Let $E$ be a quadratic field
extension of $F$ with associated quadratic character $\eta$ and $\tau \in \Gal(E/F)$ be the non-trivial involution. For $i=1,2$, let $\chi_i$ be  Hecke characters on $E_{\BA}^\times$ such that $\chi_i\neq \chi_i^\tau:=\chi_i\circ\tau$. Then  the theta lifting  $\pi_{\chi_i}$ of $\chi_i$ is a cuspidal automorphic representation on $\GL_2(\BA)$. In this setting, the triple product $L$-function decomposes into a product of Rankin-Selberg L-functions
\[\tag{$*$} L(s,\pi \times \pi_{\chi_1} \times \pi_{\chi_2})
= L(s,\pi \times \pi_{\chi_1\chi_2}) L(s,\pi \times \pi_{\chi_1\chi_2^\tau})= L(s,\pi \times \chi_1\chi_2) L(s,\pi \times \chi_1\chi_2^\tau).\]
 Let $\mu_1=\chi_1\chi_2$, $\mu_2=\chi_1\chi_2^\tau$ and $\chi=\chi_1\boxtimes\chi_2$, $\pi_\chi=\pi_{\chi_1}\boxtimes\pi_{\chi_2}$. From now on,
\begin{itemize}
    \item assume the crucial self-dual condition
\[ \mu_1|_{\BA^\times}  \omega= \mu_2|_{\BA^\times}  \omega=1,\]
which guarantees   all the three $L$-series in $(*)$ are symmetric.
\item assume moreover the sign condition
\[ \epsilon(1/2,\pi \times \pi_{\chi} )=\epsilon(1/2,\pi \times \mu_1)= \epsilon(1/2,\pi \times \mu_2)=1. \]
\end{itemize}
Then the celebrating Waldspurger and Ichino formulae imply that:
\begin{itemize}
	\item  For $i=1,2$,  there exists a  unique quaternion algebra $B_i\supset E$ over $F$ and an irreducible automorphic representation $\pi_{B_i}$ on $B_{i,\BA}^\times$ whose Jacquet-Langlands correspondence is $\pi$  such that
		\[\Hom_{E_\BA^\times}(\pi_{B_i}\boxtimes\mu_i,\BC)=\BC.\]
  Here $E^\times$ embeds into $B_i^\times\times E^\times $ diagonally. 
The Waldspurger toric  period
	\[P_{(\pi_{B_i},\mu)}^\RW:\ \varphi\in \pi_{B_i}\mapsto \int_{[F^\times \bs E^\times]} \varphi(t)\mu_i(t)dt\in \Hom_{E_\BA^\times}(\pi_{B_i},\mu_i)
	\]
 satisfies that up to simple factors,
 \[P_{(\pi_{B_i},\mu_i)}^\RW(\varphi_1) \otimes P_{(\pi_{B_i}^\vee,\mu_i^{-1})}^\RW(\varphi_2) = L
	(1/2,\pi\times \mu_i) \alpha^{\RW}(\varphi_1,\varphi_2).\]
 Here $\pi_{B_i}^\vee$ denotes the contragredient representation of $\pi_{B_i}$ and\[\alpha^{\RW}=\prod_v \alpha_v^{\RW}\neq0\in \Hom_{E_\BA^\times}(\pi_{B_i}\boxtimes\mu_i,\BC)\times  \Hom_{E_\BA^\times}(\pi_{B_i}^\vee\boxtimes\mu_i^{-1},\BC)\] is  the product of canonical local  period.
 
 \item  There exists a  unique quaternion algebra $D$ over $F$ and an irreducible automorphic representation $\Pi_D$ over $(D_{\BA}^\times)^3$ 
 whose Jacquet-Langlands correspondence is $\pi\boxtimes\pi_\chi$
such that 
		\[\Hom_{D_\BA^\times}(\Pi_D,\BC) =\BC.\]
Here $D^\times$ embeds  into $(D^\times)^3$ diagonally. The Ichino period 
\[P_{\Pi_D}^\RI:\  \varphi \in \Pi_D\mapsto\int_{[F^\times\bs D^\times]} \varphi(h)dh \in \Hom_{D_\BA^\times}(\Pi_D,\BC)\]
satisfies that up to simple factors,
	\[P_{\Pi_D}^\RI(\varphi_1) \otimes P_{\Pi_D^\vee}^\RI(\varphi_2)  = L(1/2,\Pi) \alpha^{\RI}(\varphi_1,\varphi_2)\]
	where \[\alpha^{\RI}=\prod_v \alpha_v^{\RI}\neq0\in \Hom_{D_\BA^\times}(\Pi_D,\BC)\times \Hom_{D_\BA^\times}(\Pi_D^\vee,\BC)\] is the product of canonical local period.
\end{itemize}
Now we state our results on global periods relations.
\begin{thm}[Theorem \ref{global-see-saw}] Assume $B_1=B_2=:B$, then one has $D=M_2(F)$ and moreover for any $\varphi \in \pi$, $\phi \in \CS(B_\BA \times \BA^\times)$, 
	\[P_{(\pi_B,\mu_1)}^\RW \otimes P_{(\pi_B^\vee,\mu_2^{-1})}^\RW 
	\left( \theta_\pi(\varphi,\phi) \right) = P_{\pi \boxtimes \pi_\chi}^\RI \left( \varphi \otimes
	\theta_\chi(\phi)\right)\]  
under certain normalization of Haar measures. 
	Here $\CS(B_\BA \times \BA^\times)$ is the space of Schwartz functions, 
		\[\theta_\pi \in \Hom_{\GL_2(\BA) \times (B_\BA^\times)^2}\left(
		\pi \boxtimes \CS(B_\BA \times \BA^\times), \pi_B \boxtimes \pi_B^\vee\right)\] 
		is the Shimizu lifting, and 
		\[\theta_\chi \in \Hom_{(\GL_2(\BA)^2)^0 \times ( (E_\BA^\times)^2)^0}\left(\chi
		\boxtimes \CS(B_\BA \times \BA^\times), \pi_\chi\right)\] is the global theta lifting of
		$\chi$ (see Section \ref{global theory} for details).
\end{thm}
This results follows essentially from the see-saw principle for $D=M_2(F) = E \oplus Ej$:
	\[\xymatrix{
		&\RO(D) \ar@{-}[rd]  &\Sp_2 \times \Sp_2 \\
	&\RO(E) \times \RO(Ej) \ar@{^(->}[u] \ar@{-}[ru] &\Sp_2 \ar@{^(->}[u]}\]
A parallel decomposition result holds for irreducible {\em tempered} smooth admissible representations over local fields (see Theorem \ref{local} below). Together with  the Ichino formula and  Waldspurger formula, we deduce the following global period relation even in the case $B_1\neq B_2$ and $D$ is non-split: 
\begin{thm}[Theorem \ref{main}]
Let $\varphi = \otimes \varphi_v \in \Pi_D \boxtimes \Pi_D^\vee$, 
$\varphi_i = \otimes \varphi_{i,v} \in \pi_{B_i}\boxtimes \pi_{B_i}^\vee$, $i=1,2$ be pure tensors.
Assume for each place $v$,  $\pi_v$ is tempered and
the triple $(\varphi_v,\varphi_{1,v},\varphi_{2,v})$ is {\em admissible}   so that
\[ \alpha_{\Pi_{v}}^\RI(\varphi_v) = 
		\alpha^\RW_{(\pi_{B_1,v},\mu_{1,v})}(\varphi_{1,v})\alpha^\RW_{(\pi_{B_2,v}^\vee,\mu_{2,v}^{-1})}(\varphi_{2,v}).\]
Then
\[2 P_\Pi^\RI \otimes P_{\Pi^\vee}^\RI(\varphi) =
	P_{(\pi_{B_1},\mu_1)}^\RW \otimes P^{\RW}_{(\pi_{B_1}^\vee,\mu_1^{-1})} (\varphi_1) \cdot
P^\RW_{(\pi_{B_2}^\vee,\mu_2^{-1})}\otimes P^\RW_{(\pi_{B_2},\mu_2)} (\varphi_2).\]
\end{thm}

In the above,  the triple $(\varphi_v,\varphi_{1,v},\varphi_{2,v})$ is called {\em admissible} ( see the paragraph before Theorem \ref{main} for the precise definition) 
\begin{itemize}
    \item if $D_v$ is split,  they are given by the local theta liftings as in the following local period relation (Theorem \ref{local}).
    \item if $D_v$ is nonsplit (this is the compact case), they are given as the $D_v^\times$-invariant or the $E_v^\times$-invariant vectors. 
\end{itemize}

 Under certain ramification conditions, such period relation is  considered in \cite[Sect 8]{Hsi21} and used to deduce the rationality of Darmon points (\cite{BDRSV}).  We  plan to use our results  to study the decomposition of $p$-adic $L$-functions when  both $\pi$ and $\chi_i$  vary in $p$-adic families and in turn study the generalized Kato classes (\cite{Riv22}).  
\begin{remark}
For Hecke character $\xi_1$, $\xi_2$ over $E$ such that $\xi_1\xi_2|_{\BA^\times} \eta=1$ and \[\epsilon(1/2,\pi_{\xi_1}\times\xi_2)=\epsilon(1/2,\pi_{\xi_2}\times\xi_1)=1,\]
Using the see-saw principle, Chan \cite{Chan20} establishes an explicit  period relation between relevant Waldspurger periods, reflecting the equality
		\[L(s, \pi_{\xi_1} \times \xi_2) = L(s, \xi_1 \times \pi_{\xi_2}).\]
\end{remark}

Now we state our result on local period relation. Let $v$ be a place of $F$.
\begin{thm}[Theorem \ref{local-see-saw}] \label{local}
Assume $B_{1,v}=B_{2,v}=:B_v$, then  $D_v=M_2(F_v)$. Assume moreover $\pi_v$ is tempered. Then for any $\varphi_{1,v} \in \pi_v$, $\varphi_{2,v} \in \pi_v^\vee$, 
	$\phi_{1,v},\phi_{2,v} \in \CS(B_v \times F_v^\times)$, 
	\[\alpha_{\pi \times {\pi_\chi},v}^\RI 
		\left(\varphi_{1,v} \otimes \theta_{\chi_v}(\phi_{1,v}), \varphi_{2,v} \otimes \theta_{\chi_v^{-1}}(\phi_{2,v})\right)
	=\alpha_{(\pi_B \boxtimes \pi_B^\vee, \mu_1 \boxtimes \mu_2^{-1}),v}^\RW
\left( \theta_{\pi_v}(\varphi_{1,v},\phi_{1,v}), \theta_{\pi_v^\vee}(\varphi_{2,v},\phi_{2,v}) \right)\]
under certain normalization of Haar measures and non-degenerate pair $\langle-,-\rangle_{\pi_{\chi_v}}$ (resp. $\langle \cdot,\cdot \rangle_{\pi_{B_v} \boxtimes
\pi_{B_v}^\vee}$) on $\pi_{\chi_v}\times \pi_{\chi_v}^{\vee}$ (resp. $\pi_{B_v}\boxtimes\pi_{B_v}^\vee\times \pi_{B_v}^\vee\boxtimes\pi_{B_v}$). With respect to these data(see Section \ref{local theory} for details):
\begin{itemize}
    \item $\alpha_{\pi \times {\pi_\chi},v}^\RI$  is the canonical local Ichino period on $\pi_v\boxtimes \pi_{\chi,v}\times \pi_v^\vee\boxtimes \pi^\vee_{\chi,v}$;
    \item $\alpha_{(\pi_B \boxtimes \pi_B^\vee, \mu_1 \boxtimes \mu_2^{-1}),v}^\RW$ is the product of local Waldspurger periods on $\pi_{B_v}\boxtimes \pi_{B_v}^\vee \boxtimes \mu_{1,v}\boxtimes\mu_{2,v}^{-1}$;
 \item $(\theta_{\pi_v},\theta_{\pi_v^\vee})$ is an explicit pair of local Shimizu liftings for $(\pi_v,\pi_v^\vee)$;
    \item $(\theta_{\chi_v},\theta_{\chi_v^\vee})$ is an explicit pair of local theta liftings for $(\chi_v,\chi_v^\vee)$.
\end{itemize}
\end{thm}
In the local setting, we need to consider the case $E_v = F_v^2$ so that the relevant torus is not compact. To establish the absolute convergence of integrals in the see-saw diagram, we use the asymptotic estimation of matrix coefficients of tempered representation and    the germ expansion of hyperbolic orbital integrals (see Proposition \ref{local-see-saw abs}). 

\begin{remark} In this remark, we consider the question on the seeking of test vectors for the local Ichino period when the place $v$ 
is non-Archimedean. There are already many results for this question under various ramification conditions \cite{GP,DN,Hu}. Inspired by
the work of Hsieh \cite{Hsi21}, it seems that one may drop these conditions by reducing the question to some familiar local period
via certain period relation. 

For $i=1,2,3$, let $\sigma_i$ be an irreducible smooth admissible representations on $\GL_2(F_v)$ with central character $\omega_i$. Assume $\omega_1\omega_2\omega_3=1$ and $\epsilon(1/2,\sigma_1\times\sigma_2\times\sigma_3)=1$.  Note that
\begin{itemize}
\item when at least one of $\sigma_i$ is non-supercuspidal, the result of \cite[Lemma 3.4.2]{MV10} (see also \cite[Proposition 5.1]{Hsi21}) shows 
that the local Ichino period decomposes into the product of Rankin-Selberg zeta integrals. This is the key ingredient in \cite{Hsi21} to find test vectors.
\item when two of $\sigma_i$ are theta series from the same quadratic field extension, 
\begin{itemize}
    \item if $D_v$ is split, then Theorem \ref{local} 
shows the square of the local Ichino period decomposes into  the product of the square of the local Waldspurger periods. Then one may consider test vectors
for the local Waldspurger period, e.g. the Gross-Prasad test vector (see \cite{GP,CST}), to give test vectors for the local Ichino period;
\item if $D_v$ is nonsplit, then one may consider the test vector as in Theorem \ref{main}. 
\end{itemize} 
\end{itemize}
Consider the case $v \nmid 2$. Then all $\sigma_i$ are theta series and the remaining case is $\sigma_i=\theta_{\xi_i}$ where $\xi_i$, $i=1,2,3$ are  characters on the three distinct quadratic field extensions $E_i$ over $F_v$.   Then (See \cite[Proposition 8.7]{Pra90}), 
$$L(s,\sigma_1\times\sigma_2\times\sigma_3) = L(s, \xi) L(s, \xi^{-1})$$
where $\xi=\prod_{i=1}^3\xi_i\circ N_{L/E_i}$ for $L=E_1E_2E_3$.  However, it is not clear to us that whether there exists 
an explicit decomposition for the local Ichino period in this case so that one may reduce the seeking of
test vectors to some familiar local period as above. 
\end{remark}

\s{\bf Acknowledgement} We express our sincere gratitude to Prof. Y. Tian for his consistent encouragement.

\section{Explicit theta liftings}
Throughout this paper, we will normalize the relevant Haar measures as in \cite[Section 1.6]{YZZ}.
\subsection{Local theory} \label{local theory}

In the subsection, we shall give the local theta liftings which are explicit in the sense that
the matrix coefficients of the targets are described by the matrix coefficients of the source and the Weil representation. Let $F$ be a local field of characteristic zero and fix a nontrivial character $\psi$ on $F$. For  an even dimensional quadratic space $V$
over $F$,  there is the Weil representation $r_\psi$ of the dual pair $(\SL_2, \RO(V))$  on the  space $\CS(V)$ of Schwartz functions with respect to $\psi$. As explained in \cite[Section 2.1]{YZZ}, the Weil representation extends to a representation  $r_\psi$ of the similitudes dual pair
 $(\GL_2, \GO(V))$  on the  space $\CS(V \times F^\times)$. 
 
Let $E$ be an \'etale quadratic algebra over $F$ and $B\supset E$ be a quaternion algebra over $F$. Let $-$ be the main involution on $B$ and fix a decomposition 
$$B=V_1\oplus V_2,\quad V_1=E,\quad V_2=Ej$$ where $jtj^{-1} = \bar{t}$ for any $t \in E$. Attached to these data, there are two dual pairs 
$(\GL_2, \GO(B))$ and $((\GL_2\times \GL_2)^0, 
(\GO(E)\times \GO(Ej))^0)$, where the superscript $0$ means sharing the same similitudes. 
\s{\bf W.}
Denote the Weil  representation  of $\GL_2\times\GO(B)$ on $\CS(B \times F^\times)$  as $r^\RW = r^\RW_\psi$.

\s{\bf I.} Denote by $r_i, i=1,2$  
		the Weil representation of $\GL_2 \times \GO(V_i)$ 
		on $\CS(V_i \times F^\times)$.  Then $r_1 \boxtimes r_2$ induces
		an action of  $(\GL_2\times\GL_2)^0\times (\GO(E)\times\GO(Ej))^0$ on
		$\CS(B \times F^\times)$, denoted by $r^\RI = r^\RI_\psi$ as following: for $\phi(x,u) = \phi_1(x_1,u)\phi_2(x_2,u)$ with $\phi_i \in \CS(V_i \times F^\times)$ and $x = x_1 + x_2$ with $x_i \in V_i$,
		\[r^\RI( (g_1,g_2), (h_1,h_2)) \phi(x,u) = r_1(g_1,h_1)\phi_1(x_1,u) r_2(g_2,h_2)\phi_2(x_2,u)\]
		
The Weil representation $r_\psi^W$ and $r_\psi^I$ are compatible in the following sense:


\begin{lem}[\cite{YZZ}, Page 185]\label{lem-Weil-rep} 
Embed $(\GO(E)\times \GO(Ej))^0$ into $\GO(B)$ naturally and $\GL_2$ into $(\GL_2\times\GL_2)^0$ diagonally. Then for any $\phi \in \CS(B \times F^\times)$, $g \in \GL(2)$, $(t_1,t_2) \in (\GO(E)\times\GO(Ej))^0(F)$, one has
	\[r^\RW\left( g,(t_1,t_2) \right)\phi = r^\RI( g,(t_1,t_2))\phi.\]
\end{lem}
In the following, we shall simply denote $r^\RI$ and $r^\RW$ by $r$ if there is no confusion. 
\begin{lem}The pairing 
\[\langle-,-\rangle_r:\  r_\psi\times r_{\bar{\psi}}\to \BC;\quad (\phi_1,\phi_2)\mapsto  \int_{B} \int_{F^\times} \phi_1(x,u)\phi_2(x,u)|u|^{2}dx d^\times u \]
is $(\GL_2\times\GL_2)^0(F)\times \GO(B)(F)$-invariant and non-zero.   
\end{lem}
\begin{proof}The pairing is well-defined as the integration converges absolutely and by considering standard vectors, one finds it is non-zero. Since  $\frac{d h\cdot x}{dx}=|\nu(h)|^{2}$ for any $h\in \GO(B)$ and the similitudes character $\nu$, one can check the pairing is invariant under the $(\GL_2\times\GL_2)^0\times \GO(B)$-action using the explicit description given in \cite[Section 2.1]{YZZ}.
\end{proof}

\s{\bf (W).}  The  $B^\times\times B^\times$-action on $B$ given by 
$(h_1,h_2)\cdot h=h_1h h_2^{-1}$
 induces an exact sequence
\[1 \lra F^\times \lra (B^\times \times B^\times) \rtimes \{1,\iota\} \lra \GO(B) \lra 1\]
where $\iota$ induces the main involution $-$ on $B$ 
and $F^\times$ embeds into
the group in the middle by $x \mapsto (x,x) \rtimes 1$. 

Let $\pi$ be an unitary irreducible $\GL_2(F)$-representation  with a nondegenerate invariant  pairing $\langle \cdot,\cdot \rangle_\pi$
on $\pi \times \pi^\vee$. 
Let $\pi_B$ be the unique (if exists) irreducible smooth $B^\times$-representation such that 
\[\Hom_{\GL_2(F) \times (B^\times \times B^\times)}\left( \pi \boxtimes r_\psi, 
	\pi_B \boxtimes \pi_B^\vee \right) = \BC.\]
 Note that the Jacuqet-Langlands correspondence of $\pi_B$ is $\pi$.
 \begin{lem}\label{Mat-coe Weil}For $i=1,2$, take $\phi_i(x,u)=\phi_i^\prime(x)f_i(u)\in \CS(B\times F^\times)$. Then for  $m=\diag\{a,b\}\in \GL_2(F)$,
 $$|\langle r(m)\phi_1,\phi_2\rangle_r|\leq C(\phi_1^\prime,\phi_2^\prime)|a/b|\int_{F^\times}|f_1((ab)^{-1}u) f_2(u) u^2|d^\times u$$
 with $$ C(\phi_1^\prime,\phi_2^\prime)=\max_{y\in B}|\phi_1^\prime(y)|\int_{B}|\phi_2^\prime(x)|dx<\infty.$$
 \end{lem}
 \begin{proof}By definition,  $$r(m)\phi(x,u)=|a/d|\phi_1^\prime(ax)f_1((ab)^{-1}u).$$
 Thus $$\langle r(m)\phi_1,\phi_2\rangle_r=|a/b|\int_{F^\times}f_1((ab)^{-1}u)f_2(u)|u^2|d^\times u\int_B \phi_1^{\prime}(ax)\phi_2^\prime(x)dx.$$
Now the desired result follows from the easy bound $$\left|\int_B \phi_1^{\prime}(ax)\phi_2^\prime(x)dx\right|\leq  C(\phi_1^\prime,\phi_2^\prime)= \sup_{x\in B}|\phi_1^\prime(ax)|\int_{B}|\phi_2^\prime(x)|dx.$$
 \end{proof}
\begin{prop}\label{explicit-local-theta-A}
	Let $\pi$ be a tempered representation on $\GL_2(F)$. Then the multilinear form 
	$$Z_\pi:\ \pi\times\pi^\vee\times r_\psi\times r_{\bar{\psi}}\to\BC$$
	$$(\varphi_1,\varphi_2,\phi_1,\phi_2)\mapsto 
		\int_{\GL_2(F)} \langle \pi(g)\varphi_1,\varphi_2 \rangle_\pi
	\langle r(g)\phi_1,\phi_2 \rangle_r dg$$
	 is well-defined and nonzero. Moreover,  there exists a pair of
	nonzero linear functionals \[(\theta_\pi,\theta_{\pi^\vee})\in
	 \Hom_{\GL_2(F) \times (B^\times \times B^\times)}
			\left(\pi \boxtimes r_\psi, \pi_B \boxtimes \pi_B^\vee \right)\times \Hom_{\GL_2(F) \times (B^\times \times B^\times)}
	\left(\pi^\vee \boxtimes r_{\bar{\psi}}, \pi_B^\vee \boxtimes \pi_B \right)\]
	such that
	\[\langle \theta_\pi(\varphi_1, \phi_1), 
		\theta_{\pi^\vee}(\varphi_2,\phi_2) \rangle_{\pi_B \boxtimes \pi_B^\vee}
			= Z_\pi\left( \varphi_1,\phi_1,\varphi_2,\phi_2 \right).\]
\end{prop}
\begin{proof}Let $K\subset \GL_2(F)$ be the standard maximal open compact subgroup. Let  $M^+\subset \GL_2(F)$ be the subset consisting of diagonal matrices of the form 
$$\begin{cases} 
\diag\{\varpi^a,\varpi^b\},\ b\leq a\in \BZ, & \mathrm{if}\ F\ \mathrm{non}-\mathrm{Archimedean};\\
\diag\{a,b\}, 0<a\leq b\in \BR, & \mathrm{if}\ F\ \mathrm{Archimedean}.
  \end{cases}  $$    
  By the Cartan decomposition $\GL_2(F)=KM^+K$, the integration $Z_\pi(\varphi_1,\phi_1,\varphi_2,\phi_2)$ is absolutely convergent if and only if the integral
  $$\int_{M^+}\mu(m)\langle\pi(m)\varphi_1,\varphi_2\rangle_\pi\langle r(m)\phi_1,\phi_2\rangle_r dm;\quad  \mu(m):= \frac{\Vol(KmK)}{\Vol(K)}$$
  converges absolutely. 
Let $\delta$ denote the modulus character $$\delta:\ M^+\to \BR;\quad \diag\{a,b\}\mapsto \frac{|a|}{|b|}$$
and $\sigma$ denote the height function
$$\sigma: M^+\to \BR;\quad \diag\{a,b\}\mapsto\max\{\log|a|,\log|a^{-1}|, \log|b|,\log|b^{-1}|\}.$$
Then (see \cite{Sil79})
\begin{itemize}
    \item there exists a constant $A$ such that  $$|\mu(m)|< A\delta^{-1}(m),$$
    \item ($\pi$ is tempered) there are constants $B,C$ such that 
$$|\langle \pi(m)\varphi_1,\varphi_2\rangle_\pi|\leq B\delta(m)^{1/2}(1+\sigma(m))^C.$$
\end{itemize}
Together with Lemma \ref{Mat-coe Weil}, one deduces the desired  absolute convergence. 

Note that for $\phi_i(x,u)=\phi^\prime_i(x)f_i^\prime(u)$, $Z_\pi(\varphi_1,\phi_1,\varphi_2,\phi_2)$ equals to
\begin{align*}
    &\int_{\SL_2(F)} \int_{B}r(g)\phi_1^\prime(x)\phi_2^\prime(x)dx \int_{F^\times}\int_{F^\times}f_1(u)f_2(a)
      \langle \pi(g) \pi(\diag\{u,1\})\varphi_1, \pi(\diag\{a,1\})\varphi_2\rangle_\pi d^\times a d^\times u dg\\
      =& \int_{\SL_2(F)} \int_{B}r(g)\phi_1^\prime(x)\phi_2^\prime(x)dx 
      \langle \pi(g) \int_{F^\times}f_1(u)\pi(\diag\{u,1\})\varphi_1 d^\times u, \int_{F^\times}f_2(a)\pi(\diag\{a,1\})\varphi_2d^\times a  \rangle_\pi dg
\end{align*}
where $\omega$ is the central character of $\pi$ and 
$$f_1(u)=f_1^\prime(u)\omega^{-1}|\cdot|(u),\quad f_2(a)=f_2^\prime(a)\omega|\cdot|(a).$$
Consider  $\phi_{i,n}(x,u)=\phi^\prime_i(x)f_{i,n}^\prime$ so that 
 $f_{i,n}\in \CS(F^\times)$ satisfies $$\Supp f_{n+1}\subset \Supp f_n,\quad \cap_n \Supp f_n=\{1\},\quad \int_{F^\times} f_n d^\times u=1.$$ 
Then $$\lim_{n}Z_\pi(\varphi_1,\phi_{1,n},\varphi_2,\phi_{2,n})=\int_{\SL_2(F)}\int_Br(g)\phi_1^\prime(x)\phi_2^\prime(x)dx\langle \pi(g)\varphi_1,\varphi_2\rangle_\pi dg$$
which can be made non-zero by \cite[App A]{Xue16}. Consequently, $Z_\pi$ is non-zero.

Finally, note that \[Z_\pi\in \Hom_{\GL_2(F)\times \GL_2(F)\times \Delta(B^\times\times B^\times)}(\pi\boxtimes r_\psi \otimes \pi^\vee\boxtimes r_{\bar{\psi}},\BC),\] and one has \[ \Hom_{\GL_2(F) \times (B^\times \times B^\times)}\left( \pi \boxtimes r_\psi, 
	\pi_B \boxtimes \pi_B^\vee \right) = \BC.\]
the existence of $(\theta_\pi, \theta_{\pi^\vee})$ follows.
\end{proof}

In the following, the triple $\left(\theta_\pi,\theta_{\pi^\vee}, \langle \cdot,\cdot \rangle_{\pi_B \boxtimes
\pi_B^\vee}\right)$ will be called {\em compatible to $Z_\pi$}.

\s{\bf (I).} For $i=1,2$, the  right $E^\times$ multiplication on $E$ 
 induces an isomorphism
$$\GO(V_i)\cong E^\times \rtimes \{1,\iota_i\}$$
where $\iota_i$ induces the  conjugation $z\mapsto \bar{z}$ on $E$.

Let $\chi = \chi_1 \boxtimes \chi_2$ be a  character of $E^\times \times E^\times$. 
Denote by $\pi_\chi = \pi_{\chi_1}
\boxtimes \pi_{\chi_2}$ where for $i=1,2$, $\pi_{\chi_i}$ is the unique irreducible smooth $\GL_2(F)$-representation such that
\[\Hom_{(\GL_2(F)  \times E^\times}
\left( \chi \boxtimes r_i, \pi_{\chi_i} \right) = \BC.\]


\begin{lem}\label{Mul-free}Consider the restriction of $\chi$ to $(E^\times \times E^\times)^0$. Then there exists a unique irreducible  $(\GL_2 \times \GL_2)^0(F)$-representation $\pi_\chi^B$ such that
\[\Hom_{(\GL_2 \times \GL_2)^0(F) \times (E^\times \times E^\times)^0}
\left( \chi \boxtimes r, \pi_\chi^B \right) = \BC.\]
Moreover as $(\GL_2 \times \GL_2)^0(F)$-representations, 
   \[\pi_\chi = \bigoplus_{B} \pi_\chi^B\]
   where $B$ runs over quaternion algebras over $F$ containing $E$.
\end{lem}
\begin{proof}To simplify notations, let  $$G=(\GL(2)\times\GL(2))^0,\quad H=(\GO(E)^\times\times \GO(Ej)^\times)^0.$$
$$\GL^+_2:=\{g\in G| \det(g)\in N(E^\times)\},\quad G^+=(\GL^+_2\times\GL^+_2)^0$$

Denote by $\omega_i$  the Weil representation of $\SL_2\times \RO(V_i)$ on $\CS(V_i)$.  
By \cite[Section 3-4]{Rob}, the  $(\SL_2\times \SL_2)\times (\RO(V_1)\times \RO(V_2))$-action on $\omega:=\omega_1\boxtimes\omega_2$ extends to the group $$R:=\{(g,h)\in G\times H| \nu(g)=\nu(h)\}$$
and $r\cong \mathrm{ind}_{R}^{G\times H}\omega$. Moreover the strong Howe duality for the triple holds
$$(r^+:=\mathrm{ind}_{R}^{G^+\times H}\omega,\ G^+,\ (E^\times\times E^\times)^0). $$

By \cite[Prop 5.3]{Gan}, for any character $\xi$ of $E^\times$,  $\pi_{\xi}|_{\GL_2^+}=\pi_{\xi,V_i}^+\oplus \pi_{\xi, V_i}^-$, where $\pi_{\xi,V_i}^\pm$ are irreducible  representations such that 
\begin{itemize}
\item for the plus part $r_i^+$ of $r_i$ defined similarly as $r^+$, $$\Hom_{\GL^+_2\times E^\times}(\xi\times r_i^+, \pi_{\xi,V_i}^+)=\BC$$
    \item for the quadratic character $\eta$ associated to $E$, $\pi_{\xi,V_2}^{\pm}=\pi_{\xi,V_1}^{\eta(N(j))\pm}$.
\end{itemize} Then by \cite[Lemma 4.2]{Rob},  $\pi_{\chi_1}^{\pm}\otimes \pi_{\chi_2}^{\pm}$ are all irreducible  $G^+$-representations. Moreover, one can check 
$$\Hom_{G^+\times (E^\times E^\times)^0}(\chi\boxtimes r^+, \pi_{\chi_1}^{+}\otimes \pi_{\chi_2}^{+})=\BC.$$
By \cite[Lem 6.1]{Rob}, the strong Howe duality for the triple
$(r, G, (E^\times\times E^\times)^0)$ holds. Set $$\pi_\chi^B:=\mathrm{ind}_{G^+}^G \pi_{\chi_1}^+\otimes\pi_{\chi_2}^+=\pi_{\chi_1}^+\otimes\pi_{\chi_2}^+\oplus \pi_{\chi_1}^-\otimes\pi_{\chi_2}^-.$$
Then $\pi_\chi^B$ is irreducible, $\pi_\chi=\bigoplus_B \pi_\chi^B$ and  \[\Hom_{G \times (E^\times \times E^\times)^0}
\left( \chi \boxtimes r, \pi_\chi^B \right) = \BC.\]
\end{proof}
\begin{lem}\label{Int-I}For any $\phi_1,\phi_2\in \CS(F)$, consider the function $$I(a):\ F\to\BC;\quad a\mapsto \int_F \phi_1(ax)\phi_2(x)dx.$$
Then 
\begin{itemize}
    \item for $F$ non-Archimedean, there exists $A(a),B(a)\in \CS(F)$ such that $$I(a)=A(a)+|a|^{-1}B(a^{-1}).$$
    \item for $F$ Archimedean, there exists $A(a),B(a)\in \CS(\BR)$ such that restricted on $\BR$, $$I(a)=A(a)+|a|^{-1}B(a^{-1}).$$
\end{itemize}
\end{lem}
\begin{proof}The case $F=\BC$ follows directly from the case $F=\BR$. By the fact (see  \cite[Thm 2.2.13]{AG13}) $$\CS(F)=\CS(F^\times)+\widehat{\CS(F^\times)}$$ and the relation  $$\int_F\phi_1(ax)\hat{\phi}_2(x)dx=\int_F\hat{\phi}_1(x)\phi_2(ax)dx,$$
it suffices to show $I(a)\in \CS(F)$ if $\phi_2\in \CS(F^\times)$ and $F\neq \BC$. One can see this using integration by parts.
\end{proof}
\begin{prop}\label{explicit-local-theta-B}When $\chi$ is unitary, the bilinear form 
	$$Z_\chi:\ r_{\psi}\times r_{\bar{\psi}}\to\BC,\quad (\phi_1,\phi_2)\mapsto\int_{(E^\times \times E^\times)^0} \chi(t)
	\langle r(t)\phi_1,\phi_2 \rangle_r d^\times t $$
	 is well-defined and nonzero. Moreover,  there exists 
 a pair of
nonzero invariant linear functionals
\[(\theta_\chi,\theta_{\chi^{-1}}) \in \Hom_{(\GL_2 \times \GL_2)^0(F) \times (E^\times \times E^\times)^0}
\left(\chi \boxtimes r_\psi, \pi_\chi^B \right)\times \Hom_{(\GL_2 \times \GL_2)^0(F) \times (E^\times \times E^\times)^0}
\left(\chi^{-1} \boxtimes r_{\bar{\psi}}, \pi_{\chi^{-1}}^B \right),
\]
such that  
\[\langle \theta_\chi(\phi_1), \theta_{\chi^{-1}}(\phi_2) \rangle_{\pi_\chi^B} = 
Z_\chi\left( \phi_1,\phi_2 \right).\]
\end{prop}
\begin{proof}Without lose of generality, we may assume $\phi_i(x,u)=\phi_i^\prime(x)f_i(u)$ for $\phi_i^\prime\in \CS(B)$ and $f_i\in\CS(F^\times)$.
Then for $t\in (E^\times\times E^\times)^0$,
$$\langle r(t)\phi_1,\phi_2\rangle_r=\int_{F^\times} f_1(N(t)u)f_2(u)|u|^2\int_B \phi_1^\prime(t^{-1}x)\phi_2^\prime(x)dx d^\times u.$$
Note that for any $t^\prime\in (E^\times\times E^\times)^0$, the integration 
$$\int_{E^1\times E^1}\chi(tt^\prime)F(tt^\prime)dt d^\times t^\prime,\quad F(t):=\int_B \phi_1(t^{-1}x)\phi_2(x)dx$$
converges absolutely. This is clear when  $E/F$ is a field extension and follows  from Lemma \ref{Int-I} when $E=F\oplus F$. Thus the integration descends to  a continuous function $G(a)$, $a=N(t)$. 
Clearly the integration $$\int_{N(E^\times)}\int_{F^\times}f_1(au)f_2(u)|u^2|G(a) d^\times u d^\times a$$ 
converges absolutely and gives the desired  $Z_\chi(\phi_1,\phi_2)$.
Consequently, $Z_\chi$ is well-defined.	

We can choose suitable $f_i$ so that $$\int_{N(E^\times)}\int_{F^\times}f_1(au)f_2(u)|u^2|G(a) d^\times u d^\times a= G(1)$$
By \cite[App A]{Xue16}, one can choose $\phi_i^\prime$ so that 
$$G(1)=\int_{E^1\times E^1}\chi(t)\int_B \phi_1^\prime(t^{-1}x)\phi_2^\prime(x)dx d^\times t\neq0.$$
Consequently, $Z_\chi\neq0$. Since  \[ \Hom_{(\GL_2\times\GL_2)^0(F) \times (E^\times \times E^\times)}\left( \chi \boxtimes r, 
	\pi_\chi^B \right) = \BC,\]
the existence of $(\theta_\chi, \theta_{\chi^{-1}})$ follows immediately.
\end{proof}

In the following, the triple $\left(\theta_\chi,\theta_{\chi^\vee}, \langle \cdot,\cdot \rangle_{\pi_\chi^B}\right)$ 
will be called {\em compatible to $Z_\chi$}.

\subsection{Global theory}\label{global theory} In this subsection, let $F$ be a number field with ring of adeles $\BA$. Fix a nontrivial additive character $\psi:\ F\bs \BA\to\BC^\times$. Let $E/F$ be an \'etale quadratic extension and $B\supset E$ be a quaternion algebra over $F$. Define  the Weil representation $r^\RW$ and $r^\RI$ on
$\CS(B_\BA \times \BA^\times)$  as the restrict tensor product of the local Weil
representations. As in the local situation, we shall simply denote $r^\RI$ and $r^\RW$ by $r$ if there is no confusion.

Let $(G,H)$ be one of the two dual pairs considered above. Then for any $\phi \in \CS(B_\BA \times \BA^\times)$,  the theta function
\[\theta(g,h,\phi) = \sum_{(x,u) \in B \times F^\times} r(g,h)\phi(x,u), \quad g \in G(\BA), 
h \in H(\BA)\]
is automorphic on $G(\BA)\times H(\BA)$.

Consider the case $(G,H) = (\GL_2, \GO(B))$. 
Let $\pi$ be a cuspidal automorphic $\GL_2(\BA)$-representation. Let $\pi_B$
be the unique (if exists) irreducible automorphic representation on $B_\BA^\times$ whose Jacquet-Langlands correspondence is $\pi$. 

\begin{prop}[Global Shimizu lifting, \cite{YZZ}, Section 2.1.1]
For each $\varphi \in \pi$, $\phi \in \CS(B_\BA \times \BA^\times)$ and $h \in B_\BA^\times \times B_\BA^\times$, the integral
\[\theta_\pi(\varphi,\phi)(h) := \int_{[\GL_2]} \varphi(g) \theta(g,h,\phi) dg
\]
converges absolutely. Moreover,  $\theta_\pi(\varphi,\phi)$
are automorphic forms on $B_\BA^\times \times B_\BA^\times$ and span  $\pi_B \boxtimes \pi_B^\vee$. 
\end{prop}

Consider the case $(G,H) = ( (\GL_2 \times \GL_2)^0, (\GO(E) \times \GO(Ej))^0)$. 
Let $\chi = \chi_1 \boxtimes \chi_2$ be a Hecke character on $E_\BA^\times \times E_\BA^\times$. 
Denote by $\pi_\chi = \otimes_v \pi_{\chi_v}$ the theta lifting of $\chi$ to $\GL_2(\BA)^2$. It is
automorphic and even cuspidal if $\chi \not= \chi^\tau$. Now wiew $\chi$ as a character of $(E_\BA^\times \times E_\BA^\times)^0$.
\begin{prop} For  each $\phi \in \CS(B_\BA \times \BA^\times)$ and $g \in (\GL_2 \times \GL_2)^0(\BA)$,  the integral
\[\theta_\chi(\phi)(g) :=  \int_{[(E^\times \times E^\times)^0]} \chi(t) \theta(g,t,\phi) d^\times t\]
 converges absolutely. Moreover,  
 these functions  $\theta_\chi(\phi)$ are automorphic forms on $(\GL_2 \times \GL_2)^0(\BA)$ and  span $\pi_\chi^B = \otimes_v \pi_{\chi_v}^{B_v}
\subset \pi_\chi$. 
\end{prop}
\begin{proof}Assume $\phi(x,u)=\phi^\prime(x)f(u)$. Then for any $g\in \SL_2(\BA)$, $a\in \BA^\times$ and $t\in (E_{\BA}^\times\times E_{\BA}^\times)^0$, one has 
$$ (*)\ \theta(gd(a),t,\phi)=\sum_{u\in F^\times}(\sum_{x\in B} r(g^u)\phi^\prime_t(x))f(a^{-1}N(t)u)|a|^{-1}$$
where $g^u=d(u)^{-1}g d(u)$ and $ \phi^\prime_t(x)=\phi^\prime(t^{-1}x).$
Thus
$$\int_{[(E^\times\times E^\times)^0]}\chi(t)\theta(gd(a),t,\phi)d^\times t=\int_{[(E^\times\times E^\times)^0]}\chi(t)\sum_{u\in F^\times}\left(\sum_{x\in B}r(g^u)\phi_t(x)\right) f(a^{-1}N(t)u)|a|^{-1}d^\times t.$$
To show the integration converges absolutely, up to replacing $\phi$ by $r(gd(a))\phi$, one can assume $gd(a)=1$. In this case, the integration $$F(t):=\int_{[E^1\times E^1]}\chi(tt^\prime)\sum_{x\in B}\phi^\prime_{tt^\prime}(x)d^\times t^\prime,\quad \forall t\in (E_\BA^\times \times  E_\BA^\times)^0$$
converges absolutely. Moreover $F(t)$ descends to  a continuous function $G(N(t))$ on  $N(E^\times)\bs N(E_{\BA}^\times)$.
Since  the integration 
$$\int_{N(E^\times)\bs N(E_{\BA}^\times)} G(a)\sum_{u\in F^\times}f(au) d^\times a$$
converges absolutely,  one deduces the  integration defining $\theta_\chi(\phi)(1)$ also converges absolutely and $$\theta_\chi(1)=\int_{N(E^\times)\bs N(E_{\BA}^\times)} G(t)\sum_{u\in F^\times}f(tu) d^\times t.$$ 

Since $\theta(g,t,\phi)$ is automorphic for both $g$ and $t$, one deduces  $\theta_\chi(\phi)(g)$ is an automorphic form on $(\GL_2\times\GL_2)^0(\BA)$.
Note that by choosing $f$ properly, we can arrange $$\theta_\chi(\phi)(1)=\int_{[E^1\times E^1]}\chi(t)\sum_{x\in B}\phi^\prime_{t}(x)d^\times t=:\theta^\prime_{\chi}(\phi^\prime)$$
By standard result on the theta lifting from $[E^1\times E^1]$ to $[\SL_2\times\SL_2]$, one find choose $\phi^\prime$ properly so that $\theta_\chi(\phi)(1)\neq0$. Consequently, 
the representation $\pi_\chi^{B,\prime}$ generated by  $\theta_\chi(\phi)$ is non-zero.

  Denote by $\pi_\chi^{B,+}$ the theta lifting of $\chi|_{[E^1\times E^1]}$ to $[\SL_2\times\SL_2]$. Then by local-global compatibility of $\pi_{\chi}^{B,+}$, one finds the orbit of $\pi_{\chi}^{B,+}$ under $\GL_2(F)$-twist is finite.  By $(*)$, for each $a\in \BA^\times$, the automorphic form 
$\theta_\chi(\phi)(-d(a))$ on $[\SL_2\times\SL_2]$ is a weighted sum of   $\theta^\prime_{\chi}(\phi^\prime)(g^u)$, $u\in F^\times$. Together with  Lemma \ref{Mul-free}, we deduce $\pi_\chi^{B,\prime}=\pi_\chi^B$. 
\end{proof}

\begin{remark}
	Our definition of global and local theta liftings is the dual of the one 
	in the Howe duality (See  for example \cite{GT}). Our convention follows 
	\cite{YZZ} so that $\theta_\chi$ is the representation appearing in the Waldspurger
	formula.
\end{remark}

\section{A global seesaw identity}
In this section, let $F$ be a number field with ring of adeles $\BA$.  Let $E$ be a quadratic field extension of $F$ with associated quadratic character $\eta$. Let $\tau\in\Gal(E/F)$ be the non-trivial element. Fix a nontrivial additive character $\psi:\  F\bs\BA \to\BC^\times$.

Let $\pi$ be a cuspidal automorphic representation on $\GL_2(\BA)$ with
central character $\omega$.  Let $\chi = \chi_1 \boxtimes \chi_2$
be a  Hecke characters on $E_\BA^\times \times E_\BA^\times$. 
Let $\mu_1=\chi_1\chi_2$ and $\mu_2=\chi_1\chi_2^\tau$. 
Assume that
\[\mu_1|_{\BA^\times}\omega  = \mu_2|_{\BA^\times}\omega=1.\]
Assume moreover $\chi_i \not= \chi_i^\tau$  for $i=1,2$. Thus
 the theta series $\pi_\chi:=\pi_{\chi_1}\boxtimes\pi_{\chi_2}$ is cuspidal automorphic. 
 
The quaternion algebras in the Waldspurger formulae for $\pi\boxtimes\mu_1$, $\pi\boxtimes\mu_2$ and Ichino formula for $\pi\boxtimes\pi_\chi$  are related by the following lemma.
\begin{lem}\label{simple-lemma}
   Let $v$ be a place of $F$. Let $D_v$ be the quaternion $F_v$-algebra  such that 
   \[\Hom_{D_v^\times}(\pi_v \otimes \pi_{\chi,v},\BC) = \BC.\] 
   Let $B_{1,v}$, $B_{2,v}$ be the quaternion $F_v$-algebras  containing $E_v$ such that
   \[\Hom_{B_{1,v}^\times}\left(\pi_v \otimes \chi_{1,v}\chi_{2,v},\BC\right) = 
   \Hom_{B_{2,v}^\times}\left(\pi_v \otimes \chi_{1,v}\chi_{2,v}^\tau,\BC\right) = \BC.\] 
   Then $D_v$ is split if and only if $B_{1,v} \cong B_{2,v}$. In particular, if $D_v$ is nonsplit, then
   $E_v$ must be nonsplit. 
\end{lem}
\begin{proof}
   We have the following decomposition of the local root number for the triple product $L$-function
	\[\varepsilon(1/2, \pi_v \times \pi_{\chi,v}) = 
	\varepsilon(1/2,\pi_v \times \pi_{\mu_1,v})\varepsilon(1/2,\pi_v\times \pi_{\mu_2,v}).\]
	   By a result of Prasad and Loke, the sign of $D_v$ satisfies
	  \[\varepsilon(D_v) = \varepsilon(1/2,\pi_v \times \pi_{\chi,v}).\]
	  On the other hand, by a result of Tunnell and Saito, the signs of $B_1$ and $B_2$ satisfy
	  \[\omega_v\eta_v(-1)\varepsilon(B_{1,v}) = \varepsilon(1/2,\pi_v \times \pi_{\mu_1,v}),
	  \quad \omega_v\eta_v(-1)\varepsilon(B_{2,v}) = \varepsilon(1/2,\pi_v \times \pi_{\mu_2,v}).\]
 Thus $\varepsilon(D_v) = \varepsilon(B_{1,v})\varepsilon(B_{2,v})$ and 
	  In particular,  $B_{1,v} \cong B_{2,v}$ if and only if $D_v$ is split. 
\end{proof}

As a corollary, if there exists a quaternion algebra $B\supset E$ over $F$ such that 
\[\Hom_{E_\BA^\times}(\pi_{B}\otimes\mu_1,\BC) = 
\Hom_{E_\BA^\times}(\pi_{B}\otimes\mu_2,\BC) = \BC\] 
where $\pi_B$ is the  automorphic representation on $B_\BA^\times$ whose Jacquet-Langlands correspondence is $\pi$, then 
\[\Hom_{\GL_2(\BA)}(\pi \otimes \pi_{\chi},\BC) = \BC.\]

\begin{thm}[Global seesaw identity] \label{global-see-saw}
   Assume there exists a quaternion $F$-algebra $B\supset E$  such that
	\[\Hom_{E_\BA^\times}(\pi_{B} \otimes\mu_1,\BC)  
	= \Hom_{E_\BA^\times}(\pi_{B}\otimes \mu_2,\BC) = \BC.\] 
Then for any $\varphi \in \pi$, $\phi \in \CS(B_\BA \times \BA^\times)$, we have
\[P_{(\pi_B,\mu_1)}^\RW \otimes P_{(\pi_B^\vee,\mu_2^{-1})}^\RW 
	\left( \theta_\pi(\varphi,\phi) \right) = P_{\pi \boxtimes \pi_\chi}^\RI \left( \varphi \otimes
\theta_\chi(\phi)\right).\]  
\end{thm}
We establish Theorem \ref{global-see-saw} using  
 the following seesaw diagram of similitude groups. Write $B = E \oplus Ej$
	\[\xymatrix{
	& B^\times \times B^\times \ar[r] &\mathrm{GSO}(B) \ar@{-}[rd] &(\GL_2 \times \GL_2)^0 \\
	& E^\times \times E^\times \ar@{^(->}[u] \ar[r] &(\mathrm{GSO}(E) \times \mathrm{GSO}(Ej))^0 
	\ar@{^(->}[u] \ar@{-}[ru] 
	&\GL_2  \ar@{^(->}[u]}\]
	Here the upper horizontal map is the quotient map and 
	the lower horizontal map is 
	\[E^\times \times E^\times \lra (\mathrm{GSO}(E) \times \mathrm{GSO}(Ej))^0, \quad
	(t_1,t_2) \mapsto \left( \frac{t_1}{t_2}, \frac{t_1}{t_2^\tau} \right).\]
	Let $r^\RI$ and $r^\RW$ be the Weil representations on $\GL_2 \times \GSO(B)$ and 
	$(\GL_2 \times \GL_2)^0 \times (\GSO(E) \times \GSO(Ej))^0$. Then by Lemma \ref{lem-Weil-rep},
	$r^\RI$ and $r^\RW$ are compatible when restricted to $\GL_2 \times (\GSO(E) \times \GSO(Ej))^0$. Then
	\[
	\begin{aligned}	
	\Hom_{\GL_2(\BA)}\left(\pi_\chi \otimes \pi, \BC\right) &= \Hom_{\GL_2(\BA)}(\pi_\chi^B \otimes \pi,\BC) \\
	&= \Hom_{(E_\BA^\times \times E_\BA^\times)^0}\left( (\pi_B \boxtimes \pi_B^\vee)\otimes \chi,\BC\right) \\
	&= \Hom_{E_\BA^\times \times E_\BA^\times}\left( (\pi_B \boxtimes \pi_B^\vee) \otimes 
	(\mu_1\boxtimes \mu_2^{-1}), \BC \right).
	\end{aligned}
	\]
Both sides in the equation of Theorem \ref{global-see-saw} are elements in the
	one-dimensional space
	\[\Hom_{\GL_2(\BA) \times (E^\times \times E^\times)^0(\BA)}\left(\pi \boxtimes \CS(B_\BA \times
	\BA^\times), \chi^{-1} \right).\]
	The Theorem claims they are in fact equal.
\begin{proof}[Proof of Theorem \ref{global-see-saw}]
     
    Let $\varphi_1 \otimes \varphi_2 \in \pi_B \boxtimes \pi_B^\vee$. Write 
    \[I = P_{(\pi_B,\mu_1)}^\RW \otimes P_{(\pi_B^\vee, \mu_2^{-1})}^\RW
    (\varphi_1 \otimes \varphi_2).\]
    Assume $\varphi_1 \otimes \varphi_2= \theta_\pi(\varphi,\phi)$ for some pure tensor
    for some $\varphi \in \pi$ and $\phi \in 
    \CS(B_\BA \times \BA^\times)$. By Lemma \ref{lem-Weil-rep}, we have
    \[
    \begin{aligned}
	    I &= \iint_{[F^\times \bs E^\times]^2} \theta_\pi(\varphi,\phi)
	    (t_1,t_2) \chi_1\left(\frac{t_1}{t_2}\right)\chi_2\left(\frac{t_1}{t_2^\tau}\right) dt_1dt_2 \\
	    &= \iint_{[F^\times \bs E^\times]^2} \int_{[\GL_2]} \varphi(g) 
	    \sum_{(x,u)}  r\left( (g,g), \left(\frac{t_1}{t_2}, \frac{t_1}{t_2^\tau}\right)\right)\phi(x,u)
	    \chi_1\left(\frac{t_1}{t_2}\right)\chi_2\left(\frac{t_1}{t_2^\tau}\right) dt_1dt_2.
    \end{aligned}\]
 Note that, for any $z \in \BA^\times$
    \[r\left( (zg,zg), \left(\frac{t_1}{t_2}, \frac{t_1}{t_2^\tau}\right)\right)\phi = 
      r\left( (g,g), \left(\frac{t_1}{zt_2}, \frac{t_1}{zt_2^\tau}\right)\right)\phi. \]
    Therefore, as $\omega (\chi_1\chi_2)|_{\BA^\times} = 1$,
    \[
    \begin{aligned}
	    I &= 	\iint_{[F^\times \bs E^\times]^2} \int_{[Z \bs \GL_2]} 
	    \int_{[F^\times]} \omega(z) \varphi(g) 
	    \sum_{(x,u)}  r\left( (g,g), \left(\frac{t_1}{zt_2}, \frac{t_1}{zt_2^\tau}\right)\right)\phi(x,u)	    
	    \chi_1\left(\frac{t_1}{t_2}\right)\chi_2\left(\frac{t_1}{t_2^\tau}\right) dt_1dt_2 dg\\
	    &= \int_{[Z \bs \GL_2]} \varphi(g) \int_{[F^\times \bs E^\times]} \int_{[E^\times]}
	\sum_{(x,u)}  r\left( (g,g), \left(\frac{t_1}{t_2}, \frac{t_1}{t_2^\tau}\right)\right)\phi(x,u)	    
	\chi_1\left(\frac{t_1}{t_2}\right)\chi_2\left(\frac{t_1}{t_2^\tau}\right) dt_1dt_2 dg.
    \end{aligned}\]
    Change the variable $t_1 \mapsto t_1t_2$ and note that
    \[E^\times \times F^\times \bs E^\times \stackrel{\sim}{\lra} (E^\times \times E^\times)^0,
    \quad (t_1,t_2) \mapsto \left( t_1, t_1\frac{t_2}{t_2^\tau} \right).\]
    Therefore,
    \[
	    \begin{aligned}
		    I &= \int_{[Z \bs \GL_2]} \varphi(g) \int_{[(E^\times \times E^\times)^0]}
	\sum_{(x,u)}  r\left( (g,g), (t_1, t_2)\right)\phi(x,u)
	\chi_1\left(t_1 \right)\chi_2\left(t_2 \right) dt_1dt_2 dg \\
	    &= \int_{[Z \bs \GL_2]} \varphi(g) \theta_{\chi} (\phi)(g)dg.
	    \end{aligned}\]
    
\end{proof}

\section{A local seesaw identity}
In this section, let $F$ be a local field. Let $E/F$ be an \'etale quadratic extension and $\tau\in\Gal(E/F)$ be the non-trivial involution. Let $B\supset K$ be a quaternion $F$-algebra. 
Fix an additive character $\psi$ of $F$. Let $r=r_\psi$ be the Weil representation on $\CS(B\times F^\times)$.
Let $\pi$ be a irreducible smooth representation on $\GL_2(F)$ with central character $\omega$. 
Let $\pi_B$ be the irreducible representation on $B^\times$ whose Jacquet-Langlands correspondence is $\pi$.
Let $\chi = \chi_1 \boxtimes \chi_2$ be
a character on $(E^\times)^2$. Let $\pi_\chi=\pi_{\chi_1}\boxtimes\pi_{\chi_2}$ be the theta series on $\GL_2(F) \times \GL_2(F)$
associated to $\chi$. Let $\mu_1=\chi_1\chi_2$ and $\mu_2=\chi_1\chi_2^\tau$. 

Assume that \[ \mu_1|_{F^\times}\omega=\mu_2|_{F^\times}\omega = 1.\] Consider the following two local periods:
\begin{itemize}
	\item local Ichino period on $\pi \times \pi_\chi$: for any
	$\varphi_1 \in \pi$, $\varphi_3 \in \pi^\vee$, $\varphi_2 \in \pi_\chi$, $\varphi_4 \in \pi_{\chi^{-1}}$
	\[\alpha_{\pi \times \pi_\chi}^\RI (\varphi_1\otimes \varphi_2, \varphi_3 \otimes \varphi_4)
	:= \int_{F^\times \bs \GL_2(F)} \langle \pi(g)\varphi_1,\varphi_3 \rangle_\pi
	\langle \pi_\chi(g)\varphi_2,\varphi_4 \rangle_{\pi_\chi} dg.\]
	Here, $\langle \cdot,\cdot \rangle_\pi$
        and  $\langle \cdot,\cdot \rangle_{\pi_\chi}$ are  nonzero  invariant pairings
		on $\pi \times \pi^\vee$ and $\pi_\chi \times \pi_{\chi^{-1}}$ respectively.
	\item product of local Waldspurger periods on $\pi_B\boxtimes \pi_B^\vee$ against $\mu_1\boxtimes
		\mu_2^{-1}$: for any $\varphi_1 \in \pi_B \boxtimes \pi_B^\vee$, 
	$\varphi_2 \in \pi_B^\vee \boxtimes \pi_B$, 
		\[\alpha_{(\pi_B \boxtimes \pi_B^\vee, \mu_1 \boxtimes \mu_2^{-1})}^\RW
			(\varphi_1,\varphi_2) := \int_{(F^\times \bs E^\times)^2}
		\langle \pi_B\boxtimes\pi_B^\vee\left( t \right)\varphi_1,\varphi_2 \rangle_{\pi_B \boxtimes
		\pi_B^\vee}
		\left(\mu_1\boxtimes
	\mu_2^{-1}\right)(t) dt.\]
	Here, $\langle \cdot,\cdot \rangle_{\pi_B \boxtimes \pi_B^\vee}$ is a nonzero invariant pairing
	on $(\pi_B \boxtimes \pi_B^\vee) \times (\pi_B^\vee \boxtimes \pi_B)$.
\end{itemize}

We consider the explicit local theta liftings for $\pi$ and $\chi$. Fix a
non-degenerate invariant pairing $\langle \cdot,\cdot \rangle_\pi$ on $\pi \times \pi^\vee$ and
a pairing $\langle \cdot,\cdot \rangle_r$ on $r \times r^\vee$. Then by Proposition \ref{explicit-local-theta-A},
there exists a triple 
$\left(\theta_\pi,\theta_{\pi^\vee}, \langle \cdot,\cdot \rangle_{\pi_B \boxtimes \pi_B^\vee}\right)$ compatible
to $Z_\pi$. Similarly,  by Proposition \ref{explicit-local-theta-B}, there exists a triple 
$\left(\theta_\chi,\theta_{\chi^{-1}}, \langle \cdot,\cdot \rangle_{\pi_\chi^B}\right)$ compatible
to $Z_\pi$. Fix a non-degenerate invariant pairing $\langle \cdot,\cdot \rangle_{\pi_\chi}$ 
on $\pi_\chi \times \pi_{\chi^{-1}}$ 
which restricts to $\langle \cdot,\cdot \rangle_{\pi_\chi^B}$ on $\pi_\chi^B \times \pi_{\chi^{-1}}^B$. 
Normalize the local Ichino period and the product of local Waldspurger periods by such invariant pairings.

\begin{thm}[Local seesaw identity] \label{local-see-saw}
	Assume $\pi$ is tempered, $\chi$ is unitary and 
	\[\dim \Hom_{E^\times}(\pi_{B}, \mu_1) \cdot 
	\dim \Hom_{E^\times}(\pi_{B}, \mu_2) = 1.\]
	Then for any $\varphi_1 \in \pi$, $\varphi_2 \in \pi^\vee$, 
	$\phi_1,\phi_2 \in \CS(B \times F^\times)$, we have
	\[\alpha_{\pi \times \pi_\chi}^\RI 
		\left(\varphi_1 \otimes \theta_\chi(\phi_1), \varphi_2 \otimes \theta_{\chi^{-1}}(\phi_2)\right)
	=\alpha_{(\pi_B \boxtimes \pi_B^\vee, \mu_1\boxtimes \mu_2^{-1})}^\RW
\left( \theta_\pi(\varphi_1,\phi_1), \theta_{\pi^\vee}(\varphi_2,\phi_2) \right).\]
\end{thm}

Unlike the global case, we allow the case $E = F^2$  in the local setting. In this case,  the crucial part in the proof is the  absolutely convergence of the relevant integration on the non-compact torus. Once this is established,  the proof of Theorem \ref{local-see-saw} is exactly the same
as that for the global seesaw identity.
\begin{prop}\label{local-see-saw abs}Assume $\pi$ is tempered and $\chi=\chi_1\boxtimes\chi_2|_{(E^\times\times E^\times)^0}$ is unitary. Then 
	for any $\varphi_1,\varphi_2 \in \pi$, $\phi_1,\phi_2 \in \CS(B \times F^\times)$, the following
	integral is absolutely convergent
	\[\iint_{(F^\times \bs E^\times)^2} \int_{\GL_2(F)} 
		\langle \pi(g)\varphi_1,\varphi_2 \rangle \langle r(g,(\frac{t_1}{t_2},\frac{t_1}{\bar{t}_2}))\phi_1,\phi_2 \rangle 
	\mu_1(t_1) \mu_2^{-1}(t_2)dgdt_1dt_2.\]
\end{prop}
To prove this Proposition, we need the following result:
\begin{lem}\label{Int-II} Assume $F\neq \BC$. For any $\phi\in S(F^2)$, consider the function $$H(a):=\begin{cases}\sum_{b\in\BZ}|\phi(a\varpi^b,\varpi^{-b})|,\ &\mathrm{if}\ F\neq\BR\\
\int_{t\in \BR_+^\times}|\phi(at,t^{-1})|d^\times t,\ &\mathrm{if}\ F=\BR\end{cases}$$ 
Then  there exist a finite index set $I$, Schwartz functions  $A_i(a),B_i(a)\in \CS(F)$, $i\in I$ and a constant $C>0$ depending on $\phi$ such that $$H(a)\leq \sum_i|A_i(a)|+\left|\log|a|\right|\sum_i|B_i(a)|+C.$$
\end{lem}
\begin{proof}Without lose of generality, we may assume $\phi(x_1,x_2)=\phi_1(x_1)\phi_2(x_2)$ for $\phi_1,\phi_2\in S(F)$.
When one of $\phi_i$ belongs to $S(F^\times)$, say $\phi_2$, then $$H(a)\leq (\sup_{y\in F})|\phi_1(y)|\begin{cases} \sum_{b\in\BZ}|\phi_2(\varpi^{-b}) &\mathrm{if}\ F\neq\BR;\\
\int_{\BR_+^\times}|\phi_2(t^{-1})|d^\times t &\mathrm{if}\ F=\BR\end{cases}<\infty.$$
We claim that for any $\phi\in S(F)$, there exist finitely many  Schwartz function $\phi_i^\prime\in \CS(F^\times)$  valued in $\BR$ and $\phi^\prime\in \CS(F)$ valued in $\BR_+$ such that 
$$|\phi(x)|\leq \sum_i|\phi^\prime_i(x)|+\phi^\prime(x).$$
To see this, write $\phi(x)=\phi_a(x)+i\phi_b(x)$ with  $\phi_a,\phi_b\in S(F)$ valued in $\BR$. 
For $\epsilon>0$, take $\delta_\epsilon\in \CS(F)$ valued in $\BR_+$ such that $\delta_\epsilon(x)=1$ if $|x|\leq \epsilon/2$ and $\delta_\epsilon(x)=0$ if $|x|\geq\epsilon$.
If $\phi_a(0)\neq0$, one can choose small enough  $\epsilon>0$ such that
 $$|\phi_a(x)\delta_\epsilon(x)|\in S(F),\quad \phi_a(x)(1-\delta_\epsilon(x))\in \CS(F^\times).$$
If $\phi_a(0)=0$, choose $\epsilon>0$ small enough and apply the above construction to  $$\phi_{a,\epsilon}^+(x):=\frac{\phi_a(x)+\delta_\epsilon(x)}{2};\quad  \phi_{a,\epsilon}^-(x):=\frac{\phi_a(x)-\delta_\epsilon(x)}{2}.$$
Apply similar procedure to $\phi_b$. Then 
 the claim follows from the triangle inequality. 

 By the claim, the desired statement follows from the following well-known result (see   \cite{Cai} ) on hyperbolic orbital integral: for any $\phi\in \CS(F^2)$, there exists Schwartz functions $A(a),B(a)\in \CS(F)$ such that
$$F_\phi(a):=\begin{cases}\sum_{b\in\BZ}\phi(a\varpi^b,\varpi^{-b}),\ &\mathrm{if}\ F\neq\BR\\
\int_{t\in \BR_+^\times}\phi(at,t^{-1})d^\times t,\ &\mathrm{if}\ F=\BR\end{cases}=A(a)+B(a) \log|a|.$$
\end{proof}
\begin{proof}[Proof of Proposition \ref{local-see-saw abs}]It is equivalent to the integration
$$\int_{F^\times\bs (E^\times\times E^\times)^0}\int_{\GL_2(F)}\langle \pi(g)\varphi_1,\varphi_2\rangle\langle r(g, t)\phi_1,\phi_2\rangle \chi(t) dg d^\times t$$
converges absolutely. 
When $E/F$ is a field extension, this follows from Proposition \ref{explicit-local-theta-A}. 

Assume $E=F\oplus F$. Note that $\mu_2\bs E^1\times E^1$ has finite index in $F^\times\bs (E^\times\times E^\times)^0.$
Since $$F^\times\times F^\times\cong E^1\times E^1;\ (a_1,a_2)\mapsto \left(t(a_1):=(a_1^{-1},a_1),t(a_2)=(a_2^{-1},a_2)\right)$$
and by the
Cartan decomposition, it suffices to show that
\begin{itemize}
\item for $F$ non-archimedean, the infinite sum $$\sum_{a,b\in \BZ}\sum_{d\leq c\in\BZ}\left|\mu(\diag\{\varpi^c,\varpi^d\})\langle\pi(\diag\{\varpi^c,\varpi^d\})\varphi_1,\varphi_2\rangle\langle r(\diag\{\varpi^c,\varpi^d\},(t(\varpi^a),t(\varpi^b)))\phi_1,\phi_2\rangle\right|$$
converges absolutely;
\item for $F$ archimedean, the integration 
$$\int_{a,b\in \BR_+^{\times}}\int_{d\leq c\in\BR_+^{\times}}\mu(\diag\{c,d\})\langle\pi(\diag\{c,d\})\varphi_1,\varphi_2\rangle\langle r(\diag\{c,d\},(t(a),t(b)))\phi_1,\phi_2\rangle d^\times a d^\times b d^\times c d^\times d$$
converges absolutely.
\end{itemize}
It suffices to consider the case $\phi_i(x,u)=\phi_i^\prime(x) f_i(u)$.
In this case, for any $a,b,c,d\in F^\times$
\begin{align*}&|\langle r(\diag\{c,d\},(t(a),t(b)))\phi_1,\phi_2\rangle|\\
\leq &|c/d|\left|\int_Br(t(a),t(b))\phi_1^\prime(cx)\phi_2^\prime(x)dx\right| \left|\int_{F^\times}f_1((cd)^{-1}u) f_2(u) |u|^2d^\times u\right|
\end{align*}
  Set $L=\begin{cases} F & \mathrm{if}\ F\neq \BC;\\
  \BR & \mathrm{if}\ F=\BC\end{cases}$. By Lemma \ref{Int-I}, there exists $A_i,B_i\in \CS(L)$ such that  for $a,b,c\in L^\times$
$$\int_Br(t(a),t(b))\phi_1^\prime(cx)\phi_2^\prime(x)dx=I_1(ca)I_2(ca^{-1})I_3(cb)I_4(cb^{-1}),\quad I_i(x)=A_i(x)+|x|^{-1}B_i(x^{-1}).$$
Then by Lemma \ref{Int-II}, there exist Schwartz functions $E_{i,j}, \tilde{E}_{i,j}, F_{i,j}, \tilde{F}_{i,j}\in S(L)$, $i\in I$, $j=1,2$ and constants $C_1,C_2$ such that 
\begin{align*}
& F(c^2):=\begin{cases}\sum_{a,b\in \BZ}\left|\int_Br(t(a),t(b))\phi_1^\prime(cx)\phi_2^\prime(x)dx\right| & \mathrm{if}\ F\neq \BR,\BC\\
\int_{a,b\in \BR_+^\times}\left|\int_Br(t(a),t(b))\phi_1^\prime(cx)\phi_2^\prime(x)dx\right| & \mathrm{if}\ F= \BR\ \mathrm{or}\ \BC\end{cases}\\
\leq \prod_{j=1}^2&\left( \sum_i|E_{i,j}(c^2)|+\left|\log|c^2|\right|\sum_i|F_{i,j}(c^2)|+|c^{-2}|\sum_i|\tilde{E}_{i,j}(c^{-2})|+|c^{-2}|\left|\log|c^{-2}|\right|\sum_i|\tilde{F}_{i,j}(c^{-2})|+C_j\right)
\end{align*}
Make the variable change $x=c/d$ and $y=cd$. Note that $$\int_{F^\times} f_1(y^{-1}u)f_2(u)|u^2|d^\times u\in  \CS(F^\times).$$
Thus  by Lemma \ref{Int-II} again, there are  constants $D_0,D_1,D_2,D$ such that 
\begin{align*}
G(x):=\begin{cases}\sum_{y\in \BZ}F(x \varpi^y)|\int_{F^\times} f_1(\varpi^{-y}u)f_2(u)|u^2|d^\times u| & \mathrm{if}\ F\neq \BR,\BC\\
\int_{y\in \BR_+^\times}F(xy)|\int_{F^\times} f_1(y^{-1}u)f_2(u)|u^2|d^\times u| d^\times y  & \mathrm{if}\ F= \BR\ \mathrm{or}\ \BC\end{cases}
\leq  C+\sum_j\left| \log|x|\right|^jD_j
\end{align*}
Note that (see \cite{Sil79}) since $\pi$ is tempered, there are positive constants $P,Q$ such that 
$$|\langle \pi(\diag\{c,d\})\varphi_1,\varphi_2\rangle_\pi|\leq P\delta(\diag\{c,d\})^{1/2}(1+\sigma(\diag\{c,d\}))^Q$$ where $\delta$ is the modulus character and $\sigma$ is the height function . The desired convergence follows.
\end{proof}

\section{An identity for the squares of two periods}

In this section, 
we extend Theorem \ref{global-see-saw} to general quaterion algebras via the local seesaw identity 
(Theorem \ref{local-see-saw}), the Ichino formula and the Waldspurger formula. Let $F$ be a global field. Let $E$ be a quadratic field extension of $F$ with quadratic character
$\eta$. Let $\tau\in \Gal(E/F)$ be
the nontrivial element.   Let $\pi$ be an unitary cuspidal  automorphic representation on $\GL_2(\BA)$ with
central character $\omega$.
Let $\chi = \chi_1 \boxtimes \chi_2$ 
be an unitary  Hecke character on $(E_\BA^\times)^2$. Assume for $i=1,2$, $\chi_i \not= \chi_i^\tau$ so that
the theta series $\pi_{\chi} = \pi_{\chi_1} \boxtimes \pi_{\chi_2}$ is cuspidal. Let $\mu_1=\chi_1\chi_2$ and $\mu_2=\chi_1\chi_2^\tau$.
Assume 
\[\mu_1|_{\BA^\times} \omega = \mu_2|_{\BA^\times} \omega = 1\]
and there exists
\begin{itemize}
	\item  a quaternion $F$-algebra $D$ over $F$ such that
		\[\Hom_{D_\BA^\times}(\pi_D \otimes \pi_{\chi,D},\BC) = \BC.\]
	\item  two quaternion $F$-algebras $B_1$ and $B_2$ containing $E$ such that 
		\[\Hom_{E_\BA^\times}(\pi_{B_1} \otimes \mu_1,\BC) = 
		\Hom_{E_\BA^\times}(\pi_{B_2} \otimes \mu_2,\BC) = \BC.\]
\end{itemize}
Here $\pi_D$ stands for the automorphic representation on $D^\times$ whose Jacquet-Langlands correspondence is $\pi$. The definition for $\pi_{\chi,D}$ is similar. 

Under such assumptions, the central value of the triple $L$-function $L(s,\pi \times \pi_\chi)$
is related to the Ichino period and the Waldspurger period. Normalize the Haar measure on $\BA^\times \bs E_\BA^\times$
		 so that $\Vol([F^\times \bs E^\times]) = 2L(1,\eta)$. Normalize the Haar measure on $\BA^\times \bs B_\BA^\times$
		 so that $\Vol([F^\times \bs B^\times]) = 2$ for $B=D, B_1,B_2$.

\s{\bf The Ichino formula \cite[Theorem 1.6.1]{YZZ-TRIPLE}.} 
		 For any $\varphi = \otimes \varphi_v \in \Pi_D \boxtimes \Pi_D^\vee$, we have
		\[P_{\Pi_D}^\RI \otimes P_{\Pi_D^\vee}^\RI(\varphi)  = \frac{1}{8} \CL^\RI(\Pi) \prod_v
		\frac{\alpha_{\Pi_{D,v}}^\RI(\varphi_v)}{\CL^I(\Pi_v)}. \]
		Here, $\Pi=\pi\boxtimes \Pi_\chi$, $\Pi_D = \pi_D \boxtimes \pi_{\chi,D}$ and
		\begin{itemize}
		\item The special value of $L$-function
			\[\CL^I(\Pi) = \xi_F(2)^2 \frac{L(1/2,\Pi)}{L(1,\Pi,\ad)}.\] 
			For each $v$, $\CL^I(\Pi_v)$ is defined similarly. 
		\item The canonical local period: for $\varphi_v = 
		\varphi_{v,1} \otimes \varphi_{v,2} \in \Pi_{D,v}\boxtimes\Pi_{D,v}^\vee$,
			\[\alpha_{\Pi_{D,v}}^\RI(\varphi_v) = \int_{F_v^\times \bs D_v^\times}
				\langle \Pi_{D,v}(g_v)\varphi_{v,1}, \varphi_{v,2} \rangle_{\Pi_v} dg_v, \quad \]
	Here the Haar measure $dg_v$ and non-degenerate pairing $\langle \cdot,\cdot \rangle_{\Pi_{D,v}}$ are chosen so that 	 $\prod_vdg_v$  is the Haar measure on $\BA^\times \bs D_\BA^\times$ and  $\prod_v\langle \cdot,\cdot \rangle_{\Pi_{D,v}}$ is
		 the Petersson pairing on $\Pi_D\times\Pi_D^\vee$
			\[\langle \varphi_1,\varphi_2 \rangle_{\Pi_D} = \int_{[(F^\times \bs D^\times)^3]}
		\varphi_1(g)\varphi_2(g)dg, \quad \varphi_1 \in \Pi_D, \varphi_2 \in \Pi_D^\vee.\]
		\end{itemize}

\s{\bf The Waldspurger formula \cite[Theorem 1.4]{YZZ}.}  
        For any $\varphi_i  = \otimes_v \varphi_{i,v} \in \pi_{B_i} \boxtimes \pi_{B_i}^\vee$, we have
		\[P_{(\pi_{B_i},\mu_i)}^\RW \otimes P_{(\pi_{B_i}^\vee,\mu_i^{-1})}^\RW (\varphi_i)  = 
		\frac{1}{2} \CL^\RW(\pi_{B_i},\mu_i) \prod_v\frac{\alpha_{(\pi_{B_i,v},\mu_{i,v})}^\RW(\varphi_{i,v})}{\CL^\RW(\pi_{B_i,v},\mu_{i,v})}.\]
		Here, 
		\begin{itemize}
		\item The special value of $L$-function
			\[\CL^\RW(\pi_i,\mu_i) = \xi_F(2) \frac{L(1/2,\pi_{B_i} \times \pi_{\mu_i})}
			{L(1,\eta)L(1,\pi_{B_i},\ad)}.\] 
		For each $v$,  $\CL^\RW(\pi_{B_i,v},\mu_{i,v})$
		is defined similarly. 
		\item The canonical local period: for $\varphi_{i,v} = \varphi_{i,v}^1 \otimes \varphi_{i,v}^2 \in \pi_{B_i,v}\boxtimes\pi_{B_i,v}^\vee$,
			\[\alpha_{(\pi_{B_i,v},\mu_{i,v})}^\RW(\varphi_{i,v})= \int_{F_v^\times \bs E_v^\times}
				\langle \pi_{B_i,v}(t_v)\varphi_{i,v}^1, \varphi_{i,v}^2 \rangle_{\pi_{B_i,v}} \mu_{i,v}(t_v) dt_v
		\]
		where the Haar measure $dt_v$ and non-degenerate invariant pairing $\langle \cdot,\cdot \rangle_{\pi_{B_i,v}}$ are chosen so that $\prod_v dt_v$
			is the Haar measure on 
			$\BA^\times \bs E_\BA^\times$
		and $\prod_v\langle-,-\rangle_{\pi_{B_i,v}}$ is the Petersson pairing on $\pi_{B_i}\times \pi_{B_i}^\vee$
			\[\langle \varphi_1,\varphi_2 \rangle_{\pi_{B_i}} = \int_{[F^\times \bs B_i^\times]}
		\varphi_1(g)\varphi_2(g)dg, \quad \varphi_1 \in \pi_{B_i}, \varphi_2 \in \pi_{B_i}^\vee.\]
	\end{itemize}
Consider a triple $(\varphi_v,\varphi_{1,v},\varphi_{2,v})$
where  $\varphi_v \in \Pi_v \boxtimes \Pi_v^\vee$, 
$\varphi_{1,v} \in \pi_{B_1,v} \boxtimes \pi_{B_1,v}^\vee$, $\varphi_{2,v} \in \pi_{B_2,v}^\vee \boxtimes \pi_{B_2,v}$.
\begin{itemize}
	\item If $D_v$ is split,  then $B_{1,v} = B_{2,v} = B_v$ by Lemma \ref{simple-lemma}. Normalize the invariant pairings compatible  with those
	 in the local seesaw identity (Theorem \ref{local-see-saw}) as following.
		\begin{itemize}
			\item Fix pairings $\langle \cdot,\cdot \rangle_{r_v}$ and
				$\langle \cdot,\cdot \rangle_{\pi_v}$ to normalize $Z_{\pi_v}$ and
				$Z_{\chi_v}$.
			\item Fix the triple $\left( \theta_{\pi_v}, \theta_{\pi_v^\vee},
				\langle \cdot,\cdot \rangle_{\pi_{B_v} \boxtimes \pi_{B_v}^\vee}\right)$
				 compatible with $Z_{\pi_v}$.
			\item	Fix invariant pairings
				$\langle \cdot,\cdot \rangle_{\pi_{B_v}}$ in $\alpha_{(\pi_{B_v},\mu_{1,v})}$
				and $\langle \cdot,\cdot \rangle_{\pi_{B_v}^\vee}$ in 
				$\alpha_{(\pi_{B_v}^\vee,\mu_{2,v}^{-1})}$ such that
		\[\langle f_1 \otimes f_2, f_3 \otimes f_4 \rangle_{\pi_{B_v} \boxtimes \pi_{B_v}^\vee} = 
			\langle f_1,f_3 \rangle_{\pi_{B_v}} \langle f_2,f_4 \rangle_{\pi_{B_v}^\vee},
		\quad f_1,f_3 \in \pi_{B_v},\ f_2,f_4 \in \pi_{B_v}^\vee.\]
	 Normalize 
	$\alpha_{( \pi_{B_v}\boxtimes \pi_{B_v}^\vee, \mu_1 \boxtimes \mu_2^{-1})}$
			by $\langle \cdot,\cdot \rangle_{\pi_{B_v}\boxtimes \pi_{B_v}^\vee}$.	In particular
		\[
		\alpha^\RW_{( \pi_{B_v}\boxtimes \pi_{B_v}^\vee, \mu_1 \boxtimes \mu_2^{-1})}
	 (f_1 \otimes f_2,f_3 \otimes f_4) = 	
		\alpha^\RW_{(\pi_{B_v},\mu_{1,v})}(f_1 \otimes f_3)\alpha^\RW_{(\pi_{B_v}^\vee,\mu_{2,v}^{-1})}
		(f_2 \otimes f_4).\] 
	\item Fix the triple $\left( \theta_{\chi_v}, \theta_{\chi_v^{-1}},
		\langle \cdot,\cdot \rangle_{\pi_{\chi_v}^{B_v}}\right)$
		 compatible with $Z_{\chi_v}$. Extend $\langle \cdot,\cdot \rangle_{\pi_{\chi_v}^{B_v}}$
		to a pair $\langle \cdot,\cdot \rangle_{\pi_{\chi_v}}$ on $\pi_{\chi_v} \times \pi_{\chi_v^{-1}}$.
		Normalize the local Ichino period $\alpha_{\Pi_v}^\RI$ by $\langle\cdot,\cdot \rangle_{\pi_v}$
		and $\langle \cdot,\cdot \rangle_{\pi_{\chi_v}}$.
		\end{itemize}
		Under such normalization, the triple is called {\em admissible } if
		\[\varphi_{1,v} = \varphi_{1,v}^1 \otimes \varphi_{1,v}^2 \in \pi_{B_v} \boxtimes \pi_{B_v}^\vee, \quad
		\varphi_{2,v} = \varphi_{2,v}^1 \otimes \varphi_{2,v}^2 \in \pi_{B_v}^\vee \boxtimes \pi_{B_v}\]
  \[
		\varphi_v = \left(\varphi_v^1 \otimes \theta_{\chi_v}(\phi_v^1)\right) \otimes 
		\left(\varphi_v^2 \otimes
	\theta_{\chi_v^{-1}}(\phi_v^2)\right) \in \Pi_v \boxtimes \Pi_v^\vee\]
	for $\phi_v^1, \phi_v^2 \in \CS(B_v \times F_v^\times)$  such that
	\[\varphi_{1,v}^1 \otimes \varphi_{2,v}^1 = \theta_{\pi_{B_v}}(\varphi_v^1,\phi_v^1), \quad 
	\varphi_{1,v}^2 \otimes \varphi_{2,v}^2 = \theta_{\pi_{B_v}^\vee}(\varphi_v^2,\phi_v^2).\]
		 By
		Theorem \ref{local-see-saw}, for admissible triple In $(\varphi_v,\varphi_{1,v},\varphi_{2,v})$,
		\[ \alpha_{\Pi_{v}}^\RI(\varphi_v) = 
		\alpha^\RW_{(\pi_{B_v},\mu_{1,v})}(\varphi_{1,v})\alpha^\RW_{(\pi_{B_v}^\vee,\mu_{2,v}^{-1})}(\varphi_{2,v}).\]
	\item If $D_v$ is nonsplit, then  $B_{1,v} \not= B_{2,v}$ and $E_v$ is nonsplit. In this case, $F_v^\times \bs D_v^\times$ and $F_v^\times \bs E_v^\times$ are compact. The triple is called {\em admissible} if $\varphi_v \in (\Pi_{D,v} \boxtimes \Pi_{D,v}^\vee)^{D_v^\times \times D_v^\times}$ and
		\[ 
			\varphi_{1,v} \in (\pi_{B_1,v} \boxtimes \pi_{B_1,v}^\vee)^{(E_v^\times)^2, \mu_1^{-1} \boxtimes
		\mu_1}, \quad \varphi_{2,v} \in (\pi_{B_2,v}^\vee \boxtimes \pi_{B_2,v})^{(E_v^\times)^2, \mu_2 \boxtimes
		\mu_2^{-1}}\]
		such that
		\[ \alpha_{\Pi_{D,v}}^\RI(\varphi_v) = 
		\alpha^\RW_{(\pi_{B_1,v},\mu_{1,v})}(\varphi_{1,v})\alpha^\RW_{(\pi_{B_2,v}^\vee,\mu_{2,v}^{-1})}(\varphi_{2,v}).\]
\end{itemize}

Now by comparing the Ichino formula and the Waldspurger formula, we have the
following identity between the two period integrals.

\begin{thm}\label{main} Assume $\chi$ is unitary and  for each $v$, $\pi_v$ is tempered.
Let $\varphi = \otimes \varphi_v \in \Pi_D \boxtimes \Pi_D^\vee$, 
$\varphi_1 = \otimes \varphi_{1,v} \in \pi_{B_1} \boxtimes \pi_{B_1}^\vee$ and $\varphi_2 = \otimes \varphi_{2,v} \in \pi_{B_2}^\vee \boxtimes \pi_{B_2}$  be pure tensors such that for each place $v$, 
the triple $(\varphi_v,\varphi_{1,v},\varphi_{2,v})$ is admissible so that
\[ \alpha_{\Pi_{v}}^\RI(\varphi_v) = 
		\alpha^\RW_{(\pi_{B_1,v},\mu_{1,v})}(\varphi_{1,v})\alpha^\RW_{(\pi_{B_2,v}^\vee,\mu_{2,v}^{-1})}(\varphi_{2,v}).\]
Then
\[2 P_\Pi^\RI \otimes P_{\Pi^\vee}^\RI(\varphi) =
	P_{(\pi_{B_1},\mu_1)}^\RW \otimes P^{\RW}_{(\pi_{B_1}^\vee,\mu_1^{-1})} (\varphi_1) \cdot
P^\RW_{(\pi_{B_2}^\vee,\mu_2^{-1})}\otimes P^\RW_{(\pi_{B_2},\mu_2)} (\varphi_2).\]
			\end{thm}

\begin{remark}
	Theorem \ref{main} seems different from Theorem \ref{global-see-saw} by the scalar $2$.
	This scalar is contributed from the difference between
	the product of explicit local theta liftings used in Theorem \ref{main} and 
	the explicit global theta liftings  used in Theorem \ref{global-see-saw}.
\end{remark}



\end{document}